\documentclass[a4paper,10pt]{amsart}
\usepackage{amsmath,amsfonts,amssymb,amsthm, amsfonts, amscd, url}
\usepackage[latin1]{inputenc}
\usepackage[english]{babel}

\usepackage[cmtip,arrow]{xy}
\usepackage{pb-diagram,pb-xy}

\swapnumbers 
\theoremstyle{plain}
\newtheorem{thm}{Theorem}[section]
\newtheorem{prop}[thm]{Proposition}
\newtheorem{lemma}[thm]{Lemma}

\theoremstyle{remark}
\theoremstyle{definition}
\newtheorem{rem}[thm]{Remark}
\newtheorem{rems}[thm]{Remarks}
\newtheorem{remdef}[thm]{Remark-Definition}

\newtheorem{notas}[thm]{Notations}

\title[Spaces of symmetric real matrices]{Differential properties\\
of spaces of symmetric real matrices}

\author[A. Dolcetti]{Alberto Dolcetti}
\author[D. Pertici]{Donato Pertici}
{\small \address{Dipartimento di Matematica  e Informatica ``Ulisse Dini''\\Viale Morgagni 67/a\\50134 Firenze, ITALIA}
\email{alberto.dolcetti@unifi.it, \  http://orcid.org/0000-0001-9791-8122} 

\email{donato.pertici@unifi.it,  \   http://orcid.org/0000-0003-4667-9568}}

\begin{document}
\parindent 0pt

\selectlanguage{english}

\maketitle

\vspace*{-0.2in}

\begin{abstract}
We study the differential geometric properties of the manifold of non-singular symmetric real matrices endowed with the trace metric; in case of positive definite matrices we describe the full group of isometries.
\end{abstract}

\maketitle
\vspace*{-0.2in}

{\small \tableofcontents}

\renewcommand{\thefootnote}{\fnsymbol{footnote}}
\footnotetext{
This research was partially supported by GNSAGA-INdAM (Italy).
}
\renewcommand{\thefootnote}{\arabic{footnote}}
\setcounter{footnote}{0}

\vspace*{-0.3in}

{\small {\scshape{Keywords.}} Symmetric matrices, positive definite matrices, trace metric, symmetric (Semi-)Riemannian spaces, isometries, representations of Lie groups, inner and outer automorphisms.

\smallskip

{\small {\scshape{Mathematics~Subject~Classification~(2010):} 15B48, 53C35, 53C50, 17B40, 17B10.}

\section*{Introduction.}

In this paper we carry on with the study of the so-called trace metric $g$ on the submanifolds of $GL_n$ of non-singular real matrices of order $n \ge 2$, begun in \cite{DoPe2015} and in \cite{DoPe2016}.
Precisely we study the Semi-Riemannian manifolds $(GLSym_n(p), g)$ of the non-singular symmetric real matrices of signature $(p, n-p)$, together with their Semi-Riemannian submanifolds $(SLSym_n(p), g)$ of matrices with determinant $(-1)^{n-p}$ (Section \S\,2). In particular we find geodesics, Riemann tensor, Ricci and scalar curvature and we prove that $(GLSym_n(p), g)$ is isometric to the Semi-Riemannian product $(SLSym_n(p) \times \mathbb{R}, g \times h)$, where $h$ is the euclidean metric.

The geometry of positive definite matrices is object of interest in different frameworks (see for instance \cite{Lan1999} Ch.XII,  \cite{BhaH2006} \S2, \cite{Bha2007} Ch.6, 

\cite{MoZ2011} \S3), therefore we specify the previous results to the corresponding Riemannian manifold $(\mathcal{P}_n, g)$ and to the Riemmanian submanifold $(SL\mathcal{P}_n, g)$ of positive definite matrices with determinant $1$ (Section \S\,3). 

As an application, we determine and describe geometrically the full group of isometries of $(\mathcal{P}_n, g)$  and we interpret  many isometries as suitable symmetries inside $(\mathcal{P}_n, g)$ (Section \S\,4). At first we determine the isometries of $(SL\mathcal{P}_n, g)$ by using many arguments from the theory of symmetric Riemannian spaces and from the theory of Lie group representations; then the isometries of $(\mathcal{P}_n, g)$ are deduced from the de Rham decomposition of $(\mathcal{P}_n , g)$ as $(SL\mathcal{P}_n \times \mathbb{R}, g \times h)$.

After finishing this work, we have found a paper of Lajos Moln\'ar, where he describes the isometries of the manifold of positive definite hemitian matrices $\mathit{H}_n$ endowed with a class of metrics, which includes $g$ (see \cite{Mol2015} Thm.\,3). 
Comparing this with our results, it follows that every isometry of $(\mathcal{P}_n, g)$ is the restriction of an isometry of $(\mathit{H}_n, g)$. Anyway our methods are different and should clarify the geometric descriptions and make explicit the links with the theory of symmetric Riemannian spaces and with the theory of Lie group representations.

\smallskip

\textbf{Acknowledgement.} We thank Fabio Podest\`a for many discussions about the matter of this paper and especially for his fundamental help in suggesting and clarifying to us the tools of the theory of Lie group representations, used in this paper.

\section{Preliminary facts and recalls}

\begin{notas}\label{notazioni}
In this paper, for every integer $n\ge 2$,  we denote 

- $M_n= M_n(\mathbb{R})$: the vector space of real square matrices of order $n$;

- $Sym_n=Sym_n(\mathbb{R})$ (resp. $Sym_n^0$): the vector subspace of $M_n$ of symmetric matrices (resp. with \emph{trace} equal to $0$);

- $GL_n = GL_n(\mathbb{R})$ (respectively $GL_n^+ = GL_n^+(\mathbb{R})$ and $SL_n= SL_n(\mathbb{R})$): the multiplicative group of non-degenerate matrices in $M_n$ (respectively with positive determinant and with determinant equal to $1$);

- $GLSym_n = GL_n \cap Sym_n = \{ A \in GL_n \ / \ A = A^T \}$ ($A^T$ is the \emph{transpose} of $A$);

- $GLSym_n(p)$ ($0 \le p \le n)$: the set of matrices of $GLSym_n$ with \emph{signature} $(p, n-p)$; in particular
$GLSym_n(n)$ and $GLSym_n(0)$ are respectively the sets of positive definite and negative definite matrices; we denote $GLSym_n(n)$ also by $\mathcal{P}_n$;

- $SLSym_n(p)$ ($0 \le p \le n)$: the set of matrices of $GLSym_n(p)$ with determinant $(-1)^{n-p}$; we denote $SLSym_n(n)$ (the set of positive definite symmetric matrices with determinant $1$) also by $SL\mathcal{P}_n$;

- $J_p$ ($0 \le p \le n$): the \emph{block diagonal matrix} 
$diag(I_p, -I_{n-p})$, where $I_h$ is the identity matrix of order $h$, with the agreement that $J_n = I_n$ and $J_0 = - I_n$;

- $O_n$ (resp. $SO_n$): the group of (resp. special) orthogonal matrices;

- $O_n(p)$ (resp. $SO_n(p)$), $0 \le p \le n$: the subgroup of matrices $A \in GL_n$ (resp. $A \in GL_n^+$) such that $A\,J_p A^T = J_p$ (in particular, if $p=n$, then $O_n(n)= O_n$ and $SO_n(n)=SO_n$);

- for every $n \ge 3$, $\pi_n: Spin_n \to SO_n$: the \emph{universal covering} of $SO_n$.

\smallskip

For every \emph{connected Lie group} $\mathbf{G}$, we denote by $Aut(\mathbf{G})$, by $Inn(\mathbf{G})$ and by $Out(\mathbf{G})= Aut(\mathbf{G})/ Inn(\mathbf{G})$ the groups of \emph{automorphisms}, of \emph{inner automorphisms} and of \emph{outer automorphisms} of $\mathbf{G}$. 

Since $Spin_n$ is compact, simply connected, simple Lie group, $Out(Spin_n)$ is isomorphic to the group of symmetries of the \emph{Dynkin diagram} of its \emph{Lie algebra} (see for instance \cite{W2011} Thm.\,8.11.3 and \cite{OnVi1990} pag.\,49). If $n=2m+1\ge 3$ the corresponding Dynkin diagram is $B_m$ and if $n=2m \ge 4$ it is $D_m$ (see instance \cite{W2011} Thm.\,8.9.12); as a consequence $Out(Spin_{2m+1})$ is trivial for every $m$, while $Out(Spin_{2m}) \simeq \mathbb{Z}_2$ for $m \ne 4$ and $Out(Spin_{8})\simeq \mathbf{Dih}_3$ (the \emph{dihedral group}).  

\smallskip

As usual the \emph{commutator of} $A, B \in M_n$ is $[A,B] = AB-BA$.

For every $A \in GL_n$, $A^{-T}$ denotes the matrix $(A^T)^{-1} = (A^{-1})^T$.

For every $1 \leq i,j \le n$,  $E^{(i,j)}$ denotes the matrix in $M_n$ whose $(h,k)$-entry is $1$ if $(h,k)=(i,j)$ and $0$ otherwise.

For every $A \in M_n$ we denote the \emph{exponential mapping} by $e^A= exp(A) = I_n + \sum_{i=1}^{+\infty} \dfrac{A^i}{i!}$.

We define a $C^\infty$-tensor $g$ of type $(0,2)$ on $GL_n$, by 
\begin{center}
$g_A(V,W) = tr(A^{-1}VA^{-1}W)$
\end{center} 
($tr$ indicates the \emph{trace} of a matrix). We call \emph{trace metric} the metric induced by $g$ and will denote by $g$ also its restriction to every submanifold of $GL_n$.

Finally, for every $A \in GL_n$, we denote by $\varphi$, by $\Gamma_C$, by $\varphi_C$ (with $C \in GL_n$) and by $\psi$ the following  mappings:

- $\Gamma_C(A)=CAC^T$ (\emph{congruence by} $C$); 

- $\varphi(A)=A^{-1}$ (\emph{inversion}); 

- $\varphi_C(A) = (\Gamma_C \circ \varphi)(A) = CA^{-1}C^T$;

- $\psi(A) = |det(A)|^{-2/n} A$.

In particular, for every $A \in GL_n$, we have: $det(\psi(A)) = \dfrac{1}{det(A)}$, 

$(\varphi \circ \psi)(A) = (\psi \circ \varphi)(A) = |det(A)|^{2/n} A^{-1}$; moreover $\varphi$, $\psi$, $\varphi \circ \psi$ have always period $2$ and, if $C\in GLSym_n$, then also $\varphi_C$ has period $2$ and $\varphi_C(C)=C$.
\end{notas}

\begin{rems}\label{rem-iniz} We recall some facts which are know or easy to check.

a) The sets $GLSym_n(p)$ are the $(n+1)$ open connected components of $GLSym_n$.

b) The mapping:
$(C,A) \mapsto \Gamma_C(A) = C AC^T$, gives a left action of the group  $GL_n$ on every $GLSym_n(p)$.

The group $GL_n/\{\pm I_n\}$ \emph{acts effectively} by congruence on every $GLSym_n(p)$. Indeed, arguing on the matrices $E^{(i,j)} + E^{(j,i)}$, on $Sym_n$ it is simple to check that $\Gamma_C = \Gamma_{C'}$ if and only if $C=\pm C'$ and this suffices to conclude, being $Sym_n$ the tangent space to $GLSym_n(p)$ at any point.
 
If $A \in GLSym_n(p)$, there exists a matrix $C \in GL_n$ such that 
$\Gamma_C(A) = C AC^T =  J_p$. 
This implies that, for every $p$, both $GL_n$ and $GL_n^+$ \emph{act transitively} on $GLSym_n(p)$.

Moreover $O_n(p)$ and $SO_n(p)$ are the \emph{isotropy} subgroups at $J_p$ with respect to these actions. Hence $GLSym_n(p)$ is diffeomeorphic to both \emph {homogeneous manifolds} $GL_n/O_n(p)$ and $GL_n^+/SO_n(p)$. In particular $\mathcal{P}_n \simeq GL_n/O_n \simeq GL_n^+/SO_n$ (see for instance \cite{War1983} Thm.\,3.62).

c) Since $SL_n$ acts transitively by congruence on $SLSym_n(p)$, as above we get that $SLSym_n(p)$ is a submanifold of $GLSym_n(p)$, diffeomorphic to the homogeneous manifold $SL_n/SO_n(p)$. In particular $SL\mathcal{P}_n \simeq SL_n/SO_n$.

d) An \emph{isometry} between two Semi-Riemannian manifolds is a diffeomorphism between them preserving metric tensors. 
If $(\widehat{\mathcal{M}},\hat{g})$ is any Semi-Riemannian manifold we denote by $\mathcal{I}(\widehat{\mathcal{M}},\hat{g})$ the set of isometries of $(\widehat{\mathcal{M}},\hat{g})$. 

It is known that $\mathcal{I}(\widehat{\mathcal{M}},\hat{g})$ has the structure of Lie group (see for instance \cite{O'N1983} Ch.\,9 Thm.\,32) and we will denote by $\mathcal{I}^0(\widehat{\mathcal{M}},\hat{g})$ its connected component containing the identity.

\end{rems}

We end this section by recalling the following

\begin{prop}\label{geodeticheGL}
a) $(GL_n, g)$ is a homogeneous geodesically complete Semi-Riemanian manifold of signature 
$(\dfrac{n(n+1)}{2}, \dfrac{n(n-1)}{2})$, whose geodesics are the curves: 

$t \mapsto Ke^{tC}$ for every $C \in M_n$ and $K \in GL_n$.

b) Let $K \in GL_n$ and $X, Y, Z \in M_n$.

The Riemann curvature tensor of type $(0,4)$ of $(GL_n, g)$ at $K$ is
\begin{center}
$R_{XYZW}(K) = \dfrac{1}{4} tr([K^{-1}X, K^{-1}Y] \, [K^{-1}Z, K^{-1}W]).$
\end{center}

c) $(GL_n^+, g)$ is a symmetric Semi-Riemanian manifold and among its isometries there are the congruences and the inversion $\varphi$.
\end{prop}

\begin{proof}
See \cite{DoPe2015} Prop. 1.1, Thm. 2.1, Prop. 3.1 and Prop. 1.2.
\end{proof}

\section{The Semi-Riemannian manifolds $(GLSym_n(p), g)$}

\begin{prop}\label{base-orton}
a) $(GLSym_n(p), g)$ is a homogeneous Semi-Riemannian submanifold 

of $(GL_n, g)$ with signature $(\dfrac{n(n+1)}{2} -p(n-p), p(n-p))$ for every $p=0, \cdots , n$. 

An orthonormal basis with respect to $g_{J_p}$ of the tangent space $T_{J_p}(GLSym_n(p)) = Sym_n$ is 
\begin{center}
$\mathcal{B} = \{ E^{(i,i)} : i= 1, \cdots , n \} \cup \{ S^{(i,j)} = \dfrac{E^{(i,j)}+{E^{(j,i)}}}{\sqrt{2}} : 1 \le i < j \le n \}$. 
\end{center} 

The time-like vectors are the vectors $S^{(i,j)}$ with $1\le i \le p, p+1 \le j \le n$, the remaining vectors are space-like.

\smallskip

b) For every $p= 0, \cdots , n$, the mapping $A \mapsto -A$ 
is an isometry between the Semi-Riemannian manifolds $(GLSym_n(p), g)$ and $(GLSym_n(n-p), g)$; in particular $(\mathcal{P}_n, g)$ and $(GLSym_n(0), g)$ are isometric Riemannian manifolds.
\end{prop}

\begin{proof} a)
Let $A \in GLSym_n(p)$ and $C \in GL_n$ such that  $C AC^T = \Gamma_C(A) = J_p$. Since the restriction $\Gamma_C|_{_{GLSym_n(p)}}$ in an isometry of ($GLSym_n(p), g)$, then ($GLSym_n(p), g)$ is homogeneous for any $p = 0, \cdots , n$. Hence to prove that $(GLSym_n(p), g)$ is a Semi-Riemannian submanifold of $(GL_n, g)$, it suffices to verify that $g_{_{J_p}}$ is non-degenerate on $T_{J_p}(GSym_n(p)) = Sym_n$ with the expected signature.

Since the set $\mathcal{B}$ is clearly a basis of the vector space $Sym_n$, it suffices to compute $g_{J_p}$ on the pairs of elements of $\mathcal{B}$.

For every $X=(x_{ij}), Y=(y_{ij}) \in Sym_n$, standard computations allow to get:
$
g_{J_p}(X,Y) = tr(J_pXJ_pY)= 
 \sum_{i = 1}^n x_{ii}y_{ii} + \sum_{1 \le i < j \le p} 2   x_{ij}y_{ij} + \sum_{p+1 \le i < j \le n} 2   x_{ij}y_{ij} - \sum_{1 \le i \le p < j \le n} 2   x_{ij}y_{ij}
$.

This formula allows to conclude part (a) by direct computations.

\smallskip

b) The assertion follows by trivial checks.
\end{proof}

\begin{prop}\label{isomGS}
Fix $p \in \{0, \cdots , n\}$. The inversion $\varphi$ and the congruences $\Gamma_C$ ($C \in GL_n)$ are isometries of $(GLSym_n(p), g)$. 

If $C \in GLSym_n(p)$, then the isometry $\varphi_C = \Gamma_C \circ \varphi$ is the symmetry of $(GLSym_n(p), g)$ fixing $C$; therefore  $(GLSym_n(p), g)$ is a symmetric Semi-Riemannian manifold.
\end{prop}

\begin{proof}
The first part is trivial.
For the second one it suffices to argue on the symmetry $\varphi_C: X \mapsto C X^{-1}C^T= CX^{-1}C$ with respect to every $C \in GLSym_n$.
\end{proof}

\begin{prop}\label{tenRiemGLS}
a) For any $p= 0, \cdots , n$, $(GLSym_n(p), g)$ is a totally geodesic submanifold of $(GL_n, g)$ and its geodesics are precisely the curves of type $t \mapsto Ke^{tC}$ for every $K \in GLSym_n(p)$ and every  $C = K^{-1}V$ with $V \in Sym_n$. In particular $(GLSym_n(p), g)$ is geodesically complete.

\smallskip 

b) Let $K \in GLSym_n(p)$ and $X, Y, Z \in Sym_n$.

The Riemann curvature tensor of type $(0,4)$ of $(GLSym_n(p), g)$ at $K$ is
\begin{center}
$R_{XYZW}(K) = \dfrac{1}{4} tr([K^{-1}X, K^{-1}Y] \, [K^{-1}Z, K^{-1}W])$.
\end{center}

\end{prop}

\begin{proof}
a) Remembering \ref{geodeticheGL} (a), it suffices to check that $Ke^{tC} \in GLSym_n(p)$ for every $K \in GLSym_n(p)$, for every $C = K^{-1}V$ with $V \in Sym_n$ and for every $t \in \mathbb{R}$. 

For, by standard properties of the exponential mapping: $(Ke^{tK^{-1}V})^T =  e^{tVK^{-1}} K  = K e^{tK^{-1}V}K^{-1} K = K e^{tK^{-1}V}$. So $K e^{tK^{-1}V} \in GLSym_n(p)$.

b) It follows by \ref{geodeticheGL} (b), since $(GLSym_n(p), g)$ is a totally geodesic submanifold of $(GL_n, g)$.
\end{proof}

\begin{prop}\label{RicGLSym}
For every $p \in \{ 0, \cdots , n \}$, the Ricci curvature of $(GLSym_n(p), g) $ is 
\begin{center}
$Ric_Q (X,Z) = \dfrac{1}{4}tr(Q^{-1}X)tr(Q^{-1}Z) - \dfrac{n}{4} \, g_Q(X,Z)$
\end{center}

for every $Q \in GLSym_n(p)$ and for every $X, Z \in T_Q(GLSym_n(p))= Sym_n$ 

and the scalar curvature of $(GLSym_n(p), g)$ is 
\begin{center}
$S= - \dfrac{(n-1)n(n+2)}{8}$.
\end{center}
\end{prop}

\begin{proof}
Fixed $p \in \{ 0, \cdots , n \}$, we denote $\epsilon_i = 1$, for $i = 1, \cdots,  p$, $\epsilon_i =  -1$, for $i = p+1, \cdots,  n$, and $J= J_p$. If $X, Y \in \mathbb{R}^n$ are column vectors, we denote: $<X,Y> = X^T J Y = \sum_{i=1}^n \epsilon_i x_i y_i $, where $X= (x_1, \cdots , x_n)^T$ and $Y= (y_1, \cdots , y_n)^T$ (the \emph{Minkowski product} of $\mathbb{R}^n$ of signature $(p, n-p)$).

Given every matrix $A$, we denote by $A^h$ its $h$-th column. If $A$ is symmetric, $A^h$ is the $h$-th row too. 
From \ref{tenRiemGLS} (b), by standard computation we get:

$-4 R_{XYZW} = \\
\hspace*{0.65 in}= \sum_{i,j = 1}^n \epsilon_i \epsilon_j (<X^i,Y^j><Z^j,W^i> + <Y^i,X^j><W^j,Z^i> + \\
\hspace*{1.9 in} - <Y^i,X^j><Z^j,W^i> - <X^i,Y^j><W^j,Z^i>),
$

for every $X,Y,Z,W  \in T_J(GLSym_n(p)) =Sym_n$.

The basis $\mathcal{B}$ of \ref{base-orton} can be listed also by
$\{ W_{(i,j)} = \dfrac{E^{(i,j)}+E^{(j,i)}}{\sqrt{2(1 + \delta_{ij})}} \ / \ 1 \le i \le j \le n \}$
and so $g(W_{(i,j)},W_{(i,j)}) = \epsilon_i \epsilon_j$ for every $1 \le i \le j \le n$. 

From the previous formula, for all symmetric matrices $X$, $Z$ we get:

$4Ric_J(X,Z) = -4 \sum_{1 \le k \le m \le n} g(W_{(k,m)}, W_{(k,m)}) R_{X W_{(k,m)} Z W_{(k,m)}}= \\
= 2 \sum_{1 \le k \le m \le n} \sum_{i, j = 1}^ n \epsilon_k \epsilon_m \epsilon_i \epsilon_j (<X^i, W_{(k,m)}^j><Z^j, W_{(k,m)}^i> +\\
\hspace*{2.7 in}-<X^j, W_{(k,m)}^i><Z^j, W_{(k,m)}^i>).
$

\medskip
For every symmetric matrix $A= (a_{ij})$ we have $<A^i,W_{(k,m)}^j> = \dfrac{a_{ik} \epsilon_k \delta_{jm} + a_{im} \epsilon_m \delta_{jk}}{\sqrt{2(1 + \delta_{km})}}$.

Hence, by standard computations, if $X=(x_{ij}), Z=(z_{ij}) \in Sym_n$ we get:

$
4 Ric_J(X,Z) =  \\
=\sum_{1 \le k \le m \le n} \dfrac{1}{1+\delta_{km}} [\epsilon_k \epsilon_m x_{mk} z_{mk} + \epsilon_k \epsilon_m x_{kk} z_{mm} + \epsilon_m \epsilon_k x_{mm} z_{kk} + \epsilon_m \epsilon_k x_{km} z_{km} + \\
\hspace*{0.8 in} - \sum_{j = 1}^n (\epsilon_k \epsilon_j x_{jk} z_{jk} + \epsilon_k \epsilon_j x_{jm} z_{jk} \delta_{mk} + \epsilon_k \epsilon_j x_{jk} z_{jm} \delta_{mk} + \epsilon_m \epsilon_j x_{jm} z_{jm})]
= \\
= 
2 \sum_{k=1}^n x_{kk} z_{kk} + \sum_{1 \le k < m \le n} \epsilon_k \epsilon_m (x_{kk} z_{mm} + x_{mm} z_{kk} + 2 x_{km}z_{km}) +\\ 
\hspace*{0.2 in}- 2 \sum_{k,j =1}^n \epsilon_k \epsilon_j x_{jk} z_{jk} - \sum_{1 \le k < m \le n} \epsilon_k \sum_{j=1}^n \epsilon_j x_{jk} z_{jk} -  \sum_{1 \le k < m \le n} \epsilon_m \sum_{j=1}^n \epsilon_j x_{jm}z_{jm}.
$

Note that:
$
tr(JX)tr(JZ) + g_J(X,Z) =\\
= 2 \sum_{k = 1}^n x_{kk} z_{kk} + \sum_{1 \le k < m \le n} \epsilon_k \epsilon_m (x_{kk} z_{mm} + x_{mm} z_{kk} +2 x_{km} z_{km})
$.

Moreover: $-2 g_J(X,Z) 
-2 \sum_{k,j= 1}^n \epsilon_k \epsilon_j x_{jk} z_{jk}$.

Furthermore:

$
-(n-1) g_J(X,Z) =
- \sum_{1 \le k < m \le n} \epsilon_k \sum_{j=1}^n \epsilon_j x_{jk} z_{jk} - \sum_{1 \le k < m \le n} \epsilon_m \sum_{j=1}^n \epsilon_j x_{jm} z_{jm}
$.

Comparing this with the expression of $4Ric_J(X,Z)$, we get that:

$
4Ric_J(X,Z)
 = tr(JX)tr(JZ) + g_J(X,Z) -2 g_J(X,Z)-(n-1) g_J(X,Z) = \\
 =tr(JX)tr(JZ) - ng_J(X,Z)
$.

Hence:

$
Ric_J(X,Z) = \dfrac{1}{4} tr(JX)tr(JZ) - \dfrac{n}{4} g_J(X,Z)
$,
for every $X, Z \in T_J(GLSym_n(p))= Sym_n$.

Now let $Q$ a generic matrix of $GLSym_n(p)$. We know that there exists a non-singular $C$ such that $\Gamma_{C}(Q) = CQ C^T  = J$, i.e. $JC = C^{-T} Q^{-1}$. Since $\Gamma_C$ is an isometry of $(GSym_n(p),g)$, it preserves $Ric$, i.e. for any $X, Z \in Sym_n$:

$
Ric_Q(X,Z) = 
Ric_{\Gamma_C(Q)} (\Gamma_C(X), \Gamma_C(Z)) = Ric_J(CXC^T, CZC^T)=\\
\dfrac{1}{4} tr(JCXC^T) tr(JCZC^T)
- \dfrac{n}{4} \, tr(JCXC^T JCZC^T)=
\dfrac{1}{4}tr(Q^{-1}X)tr(Q^{-1}Z) - \dfrac{n}{4} g_Q(X,Z)
$

and this concludes the first part.

The scalar curvature $S$ of $(GLSym_n(p), g)$ is constant, because this manifold is homogeneous. So it suffices to compute $S$ at the point $J$. Now $J \in T_J(GLSym_n(p))= Sym_n$ is a space-like vector (indeed $g_J(J,J)= n$), hence we have:

$
Sym_n = T_J(GLSym_n(p)) =
Span(J) \oplus (J)^\perp 
$,
where $(J)^\perp = \{ V \in Sym_n : g_J(J,V) =0 \}$.

Let $V_1, \cdots , V_d$ with $d = \dfrac{n(n+1)}{2} -1$ be an orthonormal basis of $(J)^\perp$. 
We have:

$
S= g_J(\dfrac{J}{\sqrt{n}}, \dfrac{J}{\sqrt{n}}) Ric_J(\dfrac{J}{\sqrt{n}}, \dfrac{J}{\sqrt{n}}) + \sum_{i=1}^d g_J(V_i, V_i) Ric_J(V_i,V_i)
$.

Now from the expression of $Ric$, the latter is equal to:

$
\dfrac{1}{4} [tr(\dfrac{J^2}{\sqrt{n}})]^2 - \dfrac{n}{4} tr(\dfrac{J^4}{n}) + \sum_{i=1}^d \dfrac{1}{4} g_J(V_i, V_i) [tr (J V_i) ]^2  - \dfrac{n}{4} \sum_{i=1}^d g_J(V_i, V_i) g_J(V_i, V_i) = 
- \dfrac{n}{4} d
$,
since $tr(J V_i)= g_J(J, V_i) =0$ for every $i= 0, \cdots , d$. 

Hence:
$S=  -\dfrac{n}{4} (\dfrac{n(n+1)}{2} -1) =
-\dfrac{(n-1)n(n+2)}{8}
$.
\end{proof}

\begin{prop}\label{geodSLSym}
For every $p= 0, \cdots , n$ let us consider the set $SLSym_n(p)$.

a) $(SLSym_n(p), g)$ is a homogeneous Semi-Riemannian submanifold of $(GLSym_n(p), g)$ with signature $(\dfrac{n(n+1)}{2} -p(n-p)-1, p(n-p))$.

b) $(SLSym_n(p), g)$ is a totally geodesic submanifold of $(GLSym_n(p), g)$, whose geodesics are the curves $t \mapsto K exp(tK^{-1}V)$ for every $K \in SLSym_n(p)$ and every $V \in Sym_n$ with $tr(K^{-1}V) =0$. In particular $(SLSym_n(p),g)$ is geodesically complete.

c) The inversion $\varphi$ and the congruences $\Gamma_C$ with $det(C) = \pm 1$ are isometries of the symmetric manifold $(SLSym_n(p), g)$ and, for every $Q$, $\varphi_Q$ is the symmetry fixing $Q$.

d) $(SLSym_n(p), g)$ is an Einstein manifold with Ricci tensor $Ric = - \dfrac{n}{4} g$ and with scalar curvature $S = - \, \dfrac{(n-1)n(n+2)}{8}$.
\end{prop}

\begin{proof}
a)  Note that $SLSym_n(p) = GLSym_n(p) \cap SL_n((-1)^{n-p})$, where $SL_n((-1)^{n-p}) = \{ A \in GL_n : det(A) = (-1)^{n-p} \}$ is a submanifold of $GL_n$ of codimension $1$ such that, for every $Q \in SL_n((-1)^{n-p})$, by the well-known Jacob's formula, we have: 

$T_Q(SL_n((-1)^{n-p})) = \{ V\in M_n : tr(Q^{-1}V) = 0 \}$. 

Hence:
$T_Q(GL_n) = M_n$ is the sum of its vector subspaces $T_Q(GLSym_n(p))= Sym_n$ and 

$T_Q(SL_n((-1)^{n-p}))$, since 
$Q \in T_Q(GLSym_n(p))\  \setminus \ T_Q(SL_n((-1)^{n-p}))$. 

Thus $GLSym_n(p)$ and $SL_n((-1)^{n-p})$ intersect trasversally and therefore $SLSym_n(p) = GLSym_n(p) \cap SL_n((-1)^{n-p})$ is a submanifold of $GLSym_n(p)$ of dimension $\dfrac{n(n+1)}{2} -1$ (see \cite{Hir1976} Thm.\,3.3 p. 22).

Of course $T_Q(SLSym_n(p)) = \{ V \in Sym_n :  tr(Q^{-1}V) = 0\}$ for every $Q \in SLSym_n(p)$. With the same notations as in the proof of \ref{RicGLSym}, for $J=J_p$, we have: $T_{J}(SLSym_n(p)) = (J)^\perp = \{ V \in Sym_n : g_{J}(J, V)=0 \}$. 
Since $J$ is a space-like vector of $T_{J}(GLSym_n(p))$, the restriction of $g_{J}$ to $T_{J}(SLSym_n(p))$ is non-degenerate with signature

$(\dfrac{n(n+1)}{2} -p(n-p)-1, p(n-p))$.
Now if $Q \in SLSym_n(p)$, there exists a non-singular matrix $C$ such that $\Gamma_C(Q) = C Q C^T = J$. Hence $det(C) = \pm 1$ and $\Gamma_C$ is an isometry of $(GLSym_n(p), g)$, mapping $SLSym_n(p)$ onto itself and $Q$ to $J$. We conclude that $g_J$ and $g_Q$ have the same signature, for every $Q$.

\smallskip

b) By \ref{tenRiemGLS} (a), it suffices to check that $Qe^{tC} \in SLSym_n(p)$ for every $Q \in SLSym_n(p)$, for every $C = Q^{-1}V$ with $V \in T_Q(SLSym_n(p))$ and for every $t \in \mathbb{R}$. 

For, it suffices to compute $det(Qe^{tC})$ via the fact that $tr(Q^{-1}V)=0$.

\smallskip

c) The mappings in the statement are clearly isometries. If $Q \in SLSym_n(p)$, then $\varphi_Q = \Gamma_Q \circ \varphi$ is the expected symmetry of $(SLSym_n(p), g)$, fixing $Q$.

\smallskip

d)If $Q \in SLSym_n(p)$,, then $N = N_Q= \dfrac{Q}{\sqrt{n}} \in T_Q(GLSym_n(p))= Sym_n$ is a space-like unit vector and $T_Q(SLSym_n(p)) = (N)^\perp$. 
Since $T_Q(GLSym_n(p)) =  Span(N) \oplus (N)^\perp$, if $V_1 , \cdots , V_d$ (with $d= \dfrac{n(n+1)}{2} -1$) is an orthonormal basis of $T_Q(SLSym_n(p))$, then $N, V_1 , \cdots , V_d$ is an orthonormal basis of $T_Q(GLSym_n(p))$. 

Hence, if $X, Z \in T_Q(GLSym_n(p))$ and $Ric_Q$ is the Ricci tensor at $Q$ of $(GLSym_n(p), g)$, we have: 
$Ric_Q(X,Z) =
-(R_{X N Z N} + \sum_{i=1}^d g_Q(V_i, V_i) R_{X V_i Z V_i})=
- \sum_{i=1}^d g_Q(V_i, V_i) R_{X V_i Z V_i}$,

being $R_{X N Z N}=0$, by \ref{tenRiemGLS} (b).  

By part (b), the Riemann tensor of $(SLSym_n(p), g)$ is the restriction to $SLSym_n(p)$ of the Riemann tensor of $(GLSym_n(p), g)$. From the previous expression of $Ric_Q(X,Z)$, we deduce that the restriction to $T_Q (SLSym_n(p))$ of the Ricci tensor of $(GLSym_n(p), g)$ coincides with the Ricci tensor of $(SLSym_n(p), g)$ at $Q$. Hence, if $X, Z \in T_Q(SLSym_n(p))$, by part (a), the Ricci tensor of $(SLSym_n(p), g)$ at $Q$ is

$
Ric_Q(X,Z) =
\dfrac{1}{4} \{ tr(Q^{-1}X) tr(Q^{-1}Z) - n \ g_Q(X,Z) \} =
- \dfrac{n}{4} g_Q(X, Z)
$. 

Hence $(SLSym_n(p),g)$ is an Einstein manifold  with $Ric = - \dfrac{n}{4} g$ and with
 
$S = - \dfrac{n}{4} dim(SLSym_n(p)) = - \, \dfrac{(n-1)n(n+2)}{8}$.
\end{proof}

\begin{prop}\label{decopositione-p-n}
For every $p=0, \cdots n$, $(GLSym_n(p), g)$ is isometric to the Semi-Riemannian product manifold $(SLSym_n(p) \times \mathbb{R}, g \times h)$, with $h$ euclidean metric on $\mathbb{R}$.
\end{prop}.

\vspace*{-0.35in}

\begin{proof}
The mapping $F: (SLSym_n(p) \times \mathbb{R}, g \times h) \to (GLSym_n(p), g)$, given by 

$F(Q,x) = e^{\frac{x}{\sqrt{n}}} Q$, is invertible with inverse $F^{-1}(P) = (\dfrac{P}{\sqrt[n]{|det(P)|}}, \dfrac{ln(|det(P)|)}{\sqrt{n}})$. We can easily check that
$F$ and $F^{-1}$ are isometries
(see also \cite{DoPe2015} Thm.\,4.2). 
\end{proof}

\section{The Riemanniann manifold $(\mathcal{P}_n, g)$}}

\begin{prop}\label{studioP_n}
a) $(\mathcal{P}_n, g)$ and $(SL\mathcal{P}_n , g)$ are complete, simply connected, homogeneous, Riemannian manifolds with non-positive sectional curvature; moreover $(SL\mathcal{P}_n , g)$ is an Einstein manifold and $(\mathcal{P}_n, g)$ is isometric to the Riemannian product manifold $SL\mathcal{P}_n \times \mathbb{R}$.

\smallskip

b) The Ricci curvature of $(\mathcal{P}_n, g)$ at any point $Q$ is negative semi-definite and $Ric_Q(X,X)=0$ if and only if $X=\lambda Q$, for some $\lambda \in \mathbb{R}$; the scalar curvature is $-\, \dfrac{(n-1)n(n+2)}{8}$.

\smallskip

c) $(\mathcal{P}_n, g)$ and $(SL\mathcal{P}_n , g)$ are symmetric Riemannian manifolds and, for every point $Q$, the corresponding symmetry $\varphi_Q$ has $Q$ as unique fixed point.

\smallskip

d) $(SL\mathcal{P}_2, g)$ is isometric to the hyperbolic plane $\mathcal{H}_2$ (endowed with Riemannian metric having curvature $-\dfrac{1}{2}$), hence $\mathcal{P}_2$ is isometric to the Riemannian product $\mathcal{H}_2 \times \mathbb{R}$.
\end{prop}

\begin{proof}
a) Except for the sectional curvature, the statement follows from the results proved in the previous Section.

By homogeneity, it suffices to compute the sectional curvature at point $I_n$, where, by \ref{tenRiemGLS}, we have $R_{XYXY}= \dfrac{1}{4}tr([X,Y]^2)$. Since $[X,Y]$ is skew-symmetric, $tr([X,Y]^2)$ is the opposite of the sum of the squares of the entries of the matrix $[X,Y]$ (i.e. the opposite of the \emph{Frobenius norm} of $[X,Y]$) and this allows to conclude.

\smallskip

b) The assertion about the scalar curvature is in \ref{geodSLSym} (d).

Since $Q$ is symmetric positive definite, there exists a nonsingular matrix $C$ such that $Q = C C^T = \Gamma_C(I_n)$. 
Let $X \in T_Q(\mathcal{P}_n) = Sym_n$, then by proposition \ref{RicGLSym}: 

$
4 Ric_Q(X,X) = [tr(Q^{-1}X)]^2 - n \, g_Q(X,X)
=[tr(Y)]^2 - n \ tr(Y^2) 
$,

where $Y= C^{-1} X C^{-T}$ is a symmetric matrix. The statement follows from the following:

\underline{Claim}.

For every symmetric real matrix $Y$ of order $n$, we have $[tr(Y)]^2 \le n\, tr(Y^2)$ with equality if and only if $Y = \lambda I_n$ for some $\lambda \in \mathbb{R}$.

Indeed, if $\lambda_1, \cdots , \lambda_n$ are the (possibly repeated) real eigenvalues of $Y$, then 

$[tr(Y)]^2 = \sum_{i,j=1}^n \lambda_i \lambda_j \leq \sum_{i,j=1}^n \dfrac{\lambda_i^2 + \lambda_j^2}{2} = 
n \,\sum_{i=1}^n \lambda_i^2 = n\, tr(Y^2)$ 

and the equality holds if and only if $\lambda_i = \lambda_j$ for every $i,j$, i.e. if and only if $Y = \lambda I_n$ for some $\lambda \in \mathbb{R}$, being $Y$ diagonalizable.

\smallskip

c) By \ref{isomGS} and \ref{geodSLSym}, it suffices to prove the uniqueness of the fixed point and it is enough to check this for the case $Q=I_n \in \mathcal{P}_n$, because of the homogeneity.
Hence $\varphi_{I_n} = \varphi$ and so $X \in \mathcal{P}_n$ is a fixed point of $\varphi$ if and only if $X$ is  orthogonal too, therefore $X=I_n$.

\smallskip

d) Indeed, by \ref{geodSLSym}, $(SL\mathcal{P}_2, g)$ is a complete, simply connected, homogeneous, Riemannian surface and, therefore, its curvature is constant and equal to $\dfrac{S}{2}=-\dfrac{1}{2}$ (see for instance \cite{KoNo1} VI Thm.\,7.10).
\end{proof}

\begin{rem}
$SL\mathcal{P}_n \simeq SL_n/SO_n$ is an \emph{irreducibile} symmetric space (see for instance \cite{Bes1987} Table 3 p.\,315), hence $(SL\mathcal{P}_n \times \mathbb{R}, g \times h)$ is the \emph{de Rham decomposition} of the simply connected complete Riemannian manifold $(\mathcal{P}_n, g)$ (see 
\cite{KoNo1} IV Thm.\,6.2).
\end{rem}

\begin{remdef}\label{LOG}
Let $M$ be an $n \times n$ real matrix, diagonalizable over $\mathbb{R}$ with (possibly repeated) eigenvalues $\lambda_1, \cdots , \lambda_n$, all strictly positive, and $G \in GL_n$ be a matrix such that $M = G^{-1} diag(\lambda_1, \cdots , \lambda_n ) G$. 

We denote by $LOG(M)$ the matrix $G^{-1} diag(ln(\lambda_1), \cdots , ln(\lambda_n) ) G$.

$LOG(M)$ is the unique solution of the equation $exp(X) = M$ among the $n \times n$ real matrices, which are diagonalizable over $\mathbb{R}$;  the proof is contained in \cite{Hi2008} Thm.\,1.31. $LOG(M)$ is said to be the \emph{principal logarithm of} $M$. 
Therefore, for every $r \in \mathbb{R}$, it is possible to define the $r$-\emph{th power of} $M$ as $M^r = exp(r\, LOG(M))$.
\end{remdef}

\begin{prop} Let $A, B \in \mathcal{P}_n$.

a) $\gamma(t) = A \,exp(t \,LOG(A^{-1}B)) = A(A^{-1}B)^t$is the unique geodesic arc $\gamma(t): [0,1] \to (\mathcal{P}_n, g)$ such that $\gamma(0) = A$ and $\gamma(1) = B$.

\smallskip
b) The distance $d(A, B)$ between the matrices $A, B \in \mathcal{P}_n$, induced by the trace metric, is
\begin{center}
$d(A, B)= (\sum_{i=1}^n (ln\, \mu_i)^2)^{1/2}$
\end{center}
where $\mu_1, \cdots, \mu_n$ are the (possibly repeated) eigenvalues of $A^{-1}B$.
\end{prop}

\begin{proof} These results are known (see \cite{BhaH2006} \S\,2 and \cite{Bha2007} Ch.\,6 \S\,1).
We shortly prove them, by using the arguments previously developed.

The classical Theorem of Cartan-Hadamard (\cite{O'N1983} Ch.\,10 Thm.\,22) implies that there is a unique geodesic arc as in (a) and by Hopf-Rinow (\cite{O'N1983} Ch.\,5 Thm.\,21) its length gives $d(A,B)$.

By simultaneous diagonalization, there is a non-singular matrix $C$ such that $\Gamma_C(A) = CA C^T= I_n$ (i.e. $A= C^{-1} C^{-T}$) and $\Gamma_C(B)= D :=diag(\mu_1, \cdots, \mu_n)$, where $\mu_1, \cdots, \mu_n$ are the (necessarily positive) eigenvalues of $A^{-1}B$, which is diagonalizable over $\mathbb{R}$. 

Now $LOG(D) = diag(ln(\mu_1), \cdots , \ln(\mu_n))$ and so by \ref{tenRiemGLS}, the unique geodesic arc joining $I_n$ and $D$ is $\beta(t)= exp(t\,LOG(D))$, $t \in [0,1]$.

Since $\Gamma_C$ is an isometry, we get:
$d(A,B) = d(I_n, D) = length(\beta) = [g_{_{I_n}} (\dot{\beta}(0), \dot{\beta}(0))]^{1/2} = [g_{_{I_n}} (LOG(D), LOG(D))]^{1/2} = (\sum_{i=1}^n (ln\, \mu_i)^2)^{1/2}$.

Now the unique geodesic arc joining $A$ and $B$ is 

$\gamma(t) = \Gamma_{C^{-1}} (\beta(t)) = C^{-1} \, exp(t \, LOG(D)) C^{-T} =
A \,exp(t\,C^T LOG(D) C^{-T})$.

Note that, by \ref{LOG}, $C^T LOG(D) C^{-T} = LOG(A^{-1}B)$. This allows to conclude.
\end{proof}

\begin{rem}
The description of the full group of isometries of $(\mathcal{P}_n, g)$ is in \S4. For now we recall that inversion and congruences are isometries of $(\mathcal{P}_n, g)$, so $GL_n$ acts transitively by congruences on the Riemannian manifold $(\mathcal{P}_n, g)$ (remember \ref{rem-iniz} and \ref{isomGS}).

Moreover it is possible to prove that for every pair $A, B \in \mathcal{P}_n$ there is a unique matrix $X \in \mathcal{P}_n$ such that $\Gamma_X(A) = B$ and that $X$ is the \emph{geometric mean} of the matrices $A^{-1}$ and $B$, i.e. the \emph{midpoint} of the unique geodesic joining $A^{-1}$ and $B$ in the manifold $(\mathcal{P}_n, g)$ (see for instance \cite{Bha2007} p.\,11, pp.\,106--107 and p.\,206). In particular the homogeneity of $(\mathcal{P}_n, g)$ can be obtained by means of congruences associated to positive definite matrices. 

\end{rem}

\begin{prop}
Fixed a matrix $U \in O_n$, we denote:
$
\hat{\mathcal{L}}_U = \{ UQ \in GL_n : Q \in \mathcal{P}_n \}
$
and
$
\hat{\mathcal{R}}_U = \{ QU \in GL_n : Q \in \mathcal{P}_n \}
$.
Then 
$\hat{\mathcal{L}}_U = \hat{\mathcal{R}}_U$ and $GL_n = \bigcup_{U \in O_n} \hat{\mathcal{L}}_U$ is a foliation of $(GL_n,g)$, whose leaves are isometric to $(\mathcal{P}_n,g)$ and totally geodesic in $(GL_n,g)$.
\end{prop}

\begin{proof}
From the \emph{polar decomposition} (see for instance \cite{HoJ1985} Thm.\,7.3.1), for every matrix $A \in GL_n$ there is a unique $U \in O_n$ and there are unique $Q, Q' \in \mathcal{P}_n$ such that $A=UQ=Q'U$ (so: $Q'= UQU^T$). This gives that each $A \in GL_n$ belongs to an unique $\hat{\mathcal{L}}_U$ and gives also the equality $\hat{\mathcal{L}}_U = \hat{\mathcal{R}}_U$. 
We get the last part of the statement, because, by \cite{DoPe2015} Prop.\,1.2, left translations are isometries of $(GL_n, g)$ (remember also \ref{tenRiemGLS} (a)).
\end{proof}

\begin{rem}
The roles of $O_n$ and of $\mathcal{P}_n$ in the previous Proposition are mutually interchangeable: indeed, in \cite{DoPe2016} Prop.\,4.5, we proved that $GL_n$ has analogous foliations: 
$
GL_n = \bigcup_{Q \in \mathcal{P}_n} \mathcal{L}_Q =  \bigcup_{Q \in \mathcal{P}_n} \mathcal{R}_Q 
$
whose leaves are all isometric to $(O_n, g)$ and totally geodesic in $(GL_n,g)$.
\end{rem}

\section{The group of isometries of $(\mathcal{P}_n, g)$}

\begin{lemma}
Let $\mathbf{G}$ be a connected Lie group and $\widetilde{\mathbf{G}}$ be its universal covering group. Then $Out(\mathbf{G})$ is isomorphic to a subgroup of $Out(\widetilde{\mathbf{G}})$.
\end{lemma}

\begin{proof}
For every $\alpha  \in Aut(\mathbf{G})$ we denote by $\tilde{\alpha}: \widetilde{\mathbf{G}} \to \widetilde{\mathbf{G}}$ the unique lift of $\alpha$ such that $\tilde{\alpha}(1_{\widetilde{\mathbf{G}}}) = 1_{\widetilde{\mathbf{G}}}$; then $\tilde{\alpha}$ is in $Aut(\widetilde{\mathbf{G}})$ and the mapping which associates to ${\alpha}$ the equivalence class of  $\tilde{\alpha}$ in $Out(\widetilde{\mathbf{G}})$ is a group homomorphism with kernel $Inn(\mathbf{G})$.
\end{proof}

\begin{prop}\label{ClassAut}
$Aut(SO_n) \simeq  SO_n$ if $n$ is odd and
$Aut(SO_n / \{\pm I_n\}) \simeq O_n / \{\pm I_n\}$ if $n \ne 8$ is even,
where the groups on the right act by conjugation.
\end{prop}

\begin{proof}
The following arguments need some facts recalled in \ref{notazioni}.

If $n$ is odd, the result follows from the previous Lemma and from the fact that $Out(Spin_n)$ is trivial.

If $n$ is even and different from $2$, then $Spin_n$ is also the universal covering of $SO_n / \{\pm I_n\}$. 
If $n\ne 8$ too, by the previous Lemma, $Out(SO_n / \{\pm I_n\})$ has at most two elements. Since any conjugation by elements of $O_n\!\setminus\!SO_n$ is an outer automorphism of $SO_n / \{\pm I_n\}$, then $Out(SO_n/ \{\pm I_n\}) \simeq \mathbb{Z}_2$ and this gives the assertion.

If $n=2$, then $SO_2/ \{\pm I_2\} \simeq S^1$ (the circle), $Aut(S^1) = Out(S^1) \simeq \mathbb{Z}_2$ (with elements the identity and the complex inversion) and so $Out(SO_2/ \{\pm I_2\}) \simeq \mathbb{Z}_2$ and this allows to conclude as above.  
\end{proof}

\begin{thm}\label{isomSLP}
A mapping $G: (SL\mathcal{P}_n, g) \to (SL\mathcal{P}_n, g)$ is an isometry if and only if there exists a matrix $X \in GL_n$ with $det(X) = \pm 1$ such that (with the notations of \ref{notazioni})
\begin{center}
$G = \Gamma_X$\ \  or \ \  $G = \Gamma_X \circ \varphi$.
\end{center}
Moreover $G$ fixes $I_n$ if and only if the matrix $X$ belongs to $O_n$.
\end{thm}

\begin{proof}
The mappings $\Gamma_X$ and $\Gamma_X \circ \varphi$ with $det(X) = \pm1$ are isometries by \ref{geodSLSym}\,(c).

For the converse, up to congruences with matrices of determinant $\pm 1$, we can assume that $G$ fixes $I_n$. Indeed $G(I_n) \in SL\mathcal{P}_n$ and so $G(I_n) = BB^T$ for some $B \in GL_n$ with $det(B) = \pm 1$. Therefore $(\Gamma_{B^{-1}} \circ G)(I_n) = I_n$.

Let $G$ be an isometry of $(SL\mathcal{P}_n, g)$ fixing $I_n$, $\mathcal{J}$ be the group of isometries of $(SL\mathcal{P}_n, g)$, $\mathcal{J}_{_{I_n}}$ be the corresponding  subgroup of isotropy at $I_n$ and $\mathcal{J}^0$, $\mathcal{J}_{_{I_n}}^0$ be their connected components containing the identity.
 
Since $(SL\mathcal{P}_n, g)$ is homogeneous Riemannian (remember \ref{studioP_n}), we have: $SL\mathcal{P}_n \simeq \mathcal{J}/ \mathcal{J}_{_{I_n}}$.

Remembering that $SL\mathcal{P}_n \simeq SL_n/SO_n$ (\ref{rem-iniz} (b) and (c)), from \cite{Helg2001} V Th.\,4.1 (i), we get that
$\mathcal{J}^0 \simeq SL_n$ if $n$ in odd and $\mathcal{J}^0 \simeq SL_n/\!\{\pm\,I_n\}$ if $n$ in even.

Indeed it is well-known that $SL_n$ is a connected simple Lie group and the actions of $ SL_n$ (if $n$ is odd) and  of $ SL_n/\{\pm\,I_n\}$ (if $n$ is even) are both effective.

From this we get that $dim(\mathcal{J})=dim(\mathcal{J}^0) = dim(SL_n)$ and therefore $dim(SO_n) = dim(\mathcal{J}_{I_n}) = dim(\mathcal{J}_{I_n}^0)$.

Let us consider the representation $\rho: O_n \to Aut(Sym_n^0)$ defined by $\rho(X) \cdot A = XAX^T$. Arguing on the matrices $E^{(i,j)} + E^{(j,i)}$ and $E^{(i,i)} - E^{(j,j)}$ for every $i \ne j$, we get that $Ker(\rho) = \{\pm\,I_n\}$.
Let us consider also the representation
$d: \mathcal{J}_{_{I_n}} \to Aut(Sym_n^0)$ defined by the differential at $I_n$ of every element of $\mathcal{J}_{_{I_n}}$ (remember that $T_{I_n} (SL\mathcal{P}_n) = Sym_n^0$). 
By \cite{O'N1983} Ch.\,3 Prop.\,62, $d$ is a faithful representation and so: $d(\mathcal{J}_{_{I_n}}^0) = (d(\mathcal{J}_{_{I_n}}))^0$ (the component of $d(\mathcal{J}_{_{I_n}})$ containing the identity).

Since congruences by orthogonal matrices are linear isometries fixing $I_n$, we get the inclusion $\rho(SO_n) \subseteq (d(\mathcal{J}_{_{I_n}}))^0$. Since these manifolds have the same dimension and are connected, by \emph{theorem of invariance of domain}, we deduce that $\rho(SO_n) = (d(\mathcal{J}_{_{I_n}}))^0$.

For a fixed $\chi \in \mathcal{J}_{_{I_n}}$, the previous equality gives: $d \chi \, \rho(SO_n) \, d \chi^{-1} = \rho(SO_n)$. Hence there exists a unique automorphism $\alpha$ of $\rho(SO_n)$ such that: 

(*) \ \ \ \ \ \ \ \ \ \ \ \ \ \ \ \ \ \ \ $d \chi \, \rho(X) \, d \chi^{-1} = \alpha(\rho(X))$ for every $X \in SO_n$.

\smallskip
\underline{Claim}. There is $Y \in O_n$ such that $\alpha(\rho(X))= \rho(Y) \rho(X) \rho(Y)^{-1}$ for every $X \in SO_n$ and so, by (*), we get $(\rho(Y)^{-1} \, d\,\chi) \rho(X)= \rho(X) (\rho(Y)^{-1} \, d\,\chi)$ for every $X \in SO_n$. 

\smallskip

Indeed $\rho(SO_n) \simeq SO_n$ if $n$ is odd and $\rho(SO_n) \simeq SO_n/\{\pm I_n\}$ if $n$ is even.

Hence, when $n\ne 8$, the claim follows by \ref{ClassAut}.

The case $n=8$ needs different arguments.
 
First of all, we note that $F= \rho \circ \pi_8: Spin_8 \to \rho(SO_8)$ is the universal covering of $\rho(SO_8)$; so the automorphism $\alpha \in Aut(\rho(SO_8))$ can be lifted to a unique $\tilde{\alpha} \in Aut(Spin_8)$ such that $F \circ \tilde{\alpha} = \alpha \circ F$.

As recalled in \ref{notazioni}, $Out(Spin_8)$ is isomorphic to the dihedral group $\mathbf{Dih}_3$ and therefore it is the group of order $6$ generated by elements $\delta$ and $\gamma$ of order $2$ and $3$ respectively (see also \cite{FuHa2004} and \cite{LaMi1989} for further details). In particular we can assume that $\delta$ is the equivalence class of the lifting $\tilde{\tau}_H$ of the conjugation $\tau_H$ in $SO_8$ associated to a fixed matrix $H \in O_8\!\setminus\!SO_8$ and $\gamma$ is the equivalence class of an automorphism $\tilde{\gamma}$ of $Spin_8$ having the unit $1 \in Spin_8$ as unique fixed point in the fiber $ker(F)$ (see for instance \cite{LaMi1989} Ch.\,1 \S 8).

Therefore, up to inner automorphisms, we can assume that $\tilde{\alpha}= \tilde{\tau}_H^k \circ \tilde{\gamma}^p$ with $k=0,1$ and $p=0,1,2$. 

We prove that the unique admissible possibilities are $k=0,1$ and $p=0$.

From (*) above, we deduce that $d \chi \, F(Z) \, d \chi^{-1} = \alpha(F(Z))) = F(\tilde{\alpha}(Z))$ for every $Z \in Spin_8$ (up to inner automorphisms); this implies that $F$ and $F \circ \tilde{\alpha}$ are equivalent representations of $Spin_8$.

The cases $k=0$ and $p=1,2$ are impossible because $F$, $F \circ \tilde{\gamma}$ and $F \circ \tilde{\gamma}^2$ are non-equivalent. 

Indeed, by standard facts from Lie group representation theory (and with the help of the package \emph{LiE} \cite{LCL1992}), we get that the representations $\pi_8$, $\pi_8 \circ \tilde{\gamma}$ and $\pi_8 \circ \tilde{\gamma}^2$ correspond to \emph{maximal weights} $(1,0,0,0)$, $(0,0,1,0)$ and $(0,0,0,1)$ respectively, while their trace-free second symmetric powers are $F$, $F \circ \tilde{\gamma}$ and $F \circ \tilde{\gamma}^2$ and correspond to $(2,0,0,0)$, $(0,0,2,0)$ and $(0,0,0,2)$ respectively and so they are mutually non-equivalent.

If $k=1$ and $p=1,2$, since $\tilde{\tau}_H$ is the lifting of the automorphism $\tau_H$ of $SO_8$, we can argue as above with $\rho(H)^{-1} \circ d\chi $ instead of $d\chi$.
Therefore, up to inner automorphisms, $\tilde{\alpha} = \tilde{\tau}_H^k$ with $k=0,1$ and this allows to conclude the proof of the Claim.

\smallskip 

Since the action of $\rho(SO_n)$ on $Sym_n^0$ is irreducible (see for instance \cite{KoNo2} Ch.\,XI Prop.\,7.4\,(1) and \cite{Ha2003} Prop.\,4.5\,(1)), by \cite{KoNo1} App.\,5, Lemma\,1, we obtain:
$\rho(Y)^{-1}\, d\,\chi = a\,Id + b\,J$ ($Id$ is the identity of $Sym_n^0$), 

where $a, b \in \mathbb{R}$,  $J^2 = -Id$ and $dim(Sym_n^0)$ even in case of $b\ne 0$.

If $b$ would be non-zero, then the complexification of $Sym_n^0$ should have the eigenspaces of the complexification of $J$ as invariant subspaces with respect to the complexification of the representation; so the complexification of the representation of $\rho(SO_n)$ would be reducible, while it is actually irreducible (see \cite{Ha2003} Prop.\,4.5\,(1) and Prop.\,4.6). 

Hence $b=0$ and $\rho(Y)^{-1}\, d\,\chi = a\,Id$ and so, being an isometry, we get: $a= \pm 1$, i.e. $d\, \chi = \pm\rho(Y)$. In case of $a=1$, then $\chi(A) = YAY^T = \Gamma_Y(A)$, while in case of $a=-1$, then $\chi(A) = YA^{-1}Y^T = \Gamma_Y(A^{-1})=(\Gamma_Y \circ \varphi) (A)$ (remember for instance \cite{O'N1983} Ch.\,3 Prop.\,62). This concludes the proof.
\end{proof}

\begin{rem}\label{rem-dopo-isom-SLP} 
When $n=2$, the inversion $\varphi$ in $SL\mathcal{P}_2$ is the congruence associated to the rotation matrix  
$W=\left(\begin{array}{cc}
0 & -1 \\ 
1 & 0
\end{array}\right)$ and $\mathcal{I}(SL\mathcal{P}_2,g)$ consists only of the indicated congruences and has two connected components, with $\mathcal{I}^0(SL\mathcal{P}_2,g)$ isomorphic to $SL_2/\{\pm\,I_2\}$.

If $n \ge 3$, then $\varphi$ is not a congruence and the possibilities of the previous Theorem are mutually exclusive, moreover we get that:

- if $n$ is odd, then $\mathcal{I}(SL\mathcal{P}_n,g)$ is a Lie group with two connected components, with $\mathcal{I}^0(SL\mathcal{P}_n,g)$ isomorphic to $SL_n$;

- if $n \ne 2$ is even, then $\mathcal{I}(SL\mathcal{P}_n,g)$ is a Lie group with four connected components, with $\mathcal{I}^0(SL\mathcal{P}_n,g)$ isomorphic to $SL_n/\{\pm\,I_n\}$.
\end{rem}

\begin{prop}\label{isomP_n-prel}
A mapping $L: (\mathcal{P}_n, g) \to (\mathcal{P}_n, g)$ is an isometry if and only if
there exist an isometry $G: (SL\mathcal{P}_n, g) \to (SL\mathcal{P}_n, g)$ and a matrix $B \in GL_n$ such that

$L(A)= (det\,A)^{1/n}\,(\Gamma_B \circ G)(\dfrac{A}{(det\,A)^{1/n}})$ for every $A \in \mathcal{P}_n$ or

$L(A)= (det\,A)^{-1/n} \,(\Gamma_B \circ G)(\dfrac{A}{(det\,A)^{1/n}})$ for every $A \in \mathcal{P}_n$.
\end{prop}

\begin{proof}
By \ref{studioP_n} (a) and \ref{decopositione-p-n}, we have the isometry $F: SL\mathcal{P}_n \times \mathbb{R} \to \mathcal{P}_n$ given by 

$F(Q,x) = e^{\frac{x}{\sqrt{n}}} Q$ and $F^{-1}(P) = (\dfrac{P}{\sqrt[n]{det(P)}}, \dfrac{ln(det(P))}{\sqrt{n}})$. Hence $L: \mathcal{P}_n \to \mathcal{P}_n$ is an isometry if and only if $\overline{L} = F^{-1}\circ L \circ F$ is an isometry of $SL\mathcal{P}_n \times \mathbb{R}$.

Let  $\overline{L}^{\pm}$ be the isometries of $(SL\mathcal{P}_n \times \mathbb{R}, g \times h)$ defined by $\overline{L}^{\pm} (Q, x)= (G(Q), \pm x +b)$ where $G$ is a fixed isometry of $(SL\mathcal{P}_n, g)$ and $b$ is a fixed real number. Then $L^{\pm}=F \circ \overline{L}^{\pm} \circ F^{-1}$ are isometries of $(\mathcal{P}_n, g)$. 
After denoting $B=e^{b/(2\,\sqrt{n})} I_n$, by standard computations, we get: 
$L^{\pm}(A) = det(A)^{\pm 1/n} \,(\Gamma_B \circ G)(\dfrac{A}{(det\,A)^{1/n}})$.

For the converse, let $L$ be an isometry of $(\mathcal{P}_n, g)$. Since $L(I_n) \in \mathcal{P}_n$, there exists $B \in GL_n$ such that $L(I_n) = BB^T$ and so $H=\Gamma_{B^{-1}} \circ L$ is an isometry of $(\mathcal{P}_n, g)$ fixing $I_n$. 

Let us consider the differential: $D=dH_{I_n}: T_{I_n}(\mathcal{P}_n) \to T_{I_n}(\mathcal{P}_n)$ (remember that $T_{I_n}(\mathcal{P}_n) = Sym_n$). We want to prove that $D(I_n) = \pm I_n$.

For, $D$ preserves the metric $g$ and its Riemann tensor of type $(0,4)$ at $I_n$.
Remembering \ref{tenRiemGLS} (b), we have $tr([D(I_n), D(W)]^2) = tr([I_n, W]^2)=0$ for every $W \in Sym_n$. 
Since the bracket of symmetric matrices is skew-symmetric and since the opposite of the trace of the square of a skew-symmetric matrix is its Frobenius norm, we get that $[D(I_n), D(W)]=0$ for every $W \in Sym_n$, i.e. $[D(I_n), U]=0$ for every $U \in Sym_n$, because $D$ is bijective and therefore $D(I_n) = \lambda I_n$ for some $\lambda \in \mathbb{R}$. Since $D$ is an isometry, we get $\lambda = \pm 1$.

Now note that the space $(I_n)^\perp := \{ W \in Sym_n : g_{_{I_n}}(I_n, W) = 0\} = Sym_n^0 = T_{I_n}(SL\mathcal{P}_n)$ is invariant with respect to $D$, because $D(I_n)= \pm I_n$. Hence $D'$, the restriction of $D$ to $(I_n)^\perp$, is an isometry of $T_{I_n}(SL\mathcal{P}_n)$ with respect to the metric $g$.

Since $D$ preserves the Riemann tensor of $(\mathcal{P}_n, g)$ at $I_n$, $D'$ preserves its restriction to $T_{I_n}(SL\mathcal{P}_n)$, but this last restriction is the Riemann tensor of $(SL\mathcal{P}_n, g)$ at $I_n$, because $(SL\mathcal{P}_n, g)$ is a totally geodesic submanifold of $(\mathcal{P}_n, g)$ (remember \ref{geodSLSym} (b)).

Since $SL\mathcal{P}_n$ is simply connected, complete and symmetric (remember (\ref{studioP_n})),  by   

\cite{KoNo1} VI Cor.\,7.9, there exists a unique isometry $G$ of $(SL\mathcal{P}_n, g)$ such that $G(I_n) = I_n$ and $dG_{I_n}=D'$.

Now we denote $G^{\pm}(A, x) = (G(A), \pm x)$ for every $(A, x) \in SL\mathcal{P}_n \times \mathbb{R}$. $G^{\pm}(A,  x)$ are isometries of $(SL\mathcal{P}_n \times \mathbb{R}, g \times h)$ such that $G^{\pm}(I_n, 0) = (I_n, 0)$ and such that $dG^{\pm}_{(I_n, 0)} = (D' \times (\pm I\!d_{\,\mathbb{R}}))$. 

Easy computations show that 
$dF_{(I_n, 0)}(V, x) = \dfrac{x}{\sqrt{n}} I_n +V$ for every $x \in \mathbb{R}$ and every $V \in T_{I_n} (SL\mathcal{P}_n)=Sym_n^0$ and that
$dF^{-1}_{I_n}(W) = (W-\dfrac{tr(W)}{n} I_n, \dfrac{tr(W)}{\sqrt{n}})$, 

for every $W \in T_{I_n} (\mathcal{P}_n)=Sym_n$,
where $F$ and $F^{-1}$ are the mapping recalled above.

Now $(F \circ G^{\pm} \circ F^{-1})(I_n) = I_n = H(I_n)$ and 

$d(F \circ G^{\pm} \circ F^{-1})_{I_n}(W) = 
dF_{(I_n, 0)} (d G^{\pm}_{(I_n,0)}(W- \dfrac{tr(W)}{n} I_n, \dfrac{tr(W)}{\sqrt{n}})) = \\
=dF_{(I_n, 0)}(dH_{I_n}(W- \dfrac{tr(W)}{n} I_n), \pm \dfrac{tr(W)}{\sqrt{n}}) = \pm \dfrac{tr(W)}{n} I_n +dH_{I_n}(W- \dfrac{tr(W)}{n} I_n)=\\
=\pm \dfrac{tr(W)}{n} I_n + dH_{I_n}(W) \mp \dfrac{tr(W)}{n} I_n = dH_{I_n}(W)$ for every $W \in T_{I_n} (\mathcal{P}_n)=Sym_n$.

Therefore $F \circ G^{\pm} \circ F^{-1} = H$ (see again \cite{O'N1983} Ch.\,3 Prop.\,62). 

Now $L= \Gamma_B \circ H = \Gamma_B \circ F \circ G^{\pm} \circ F^{-1}$ and easy computations allow the get the expressions in the statement. 
\end{proof}

Remembering the definitions of $\varphi$ and $\psi$ in \ref{notazioni}, from the previous \ref{isomSLP} and \ref{isomP_n-prel}, we easily get the following

\begin{thm}\label{isomP_n}
A mapping $L: (\mathcal{P}_n, g) \to (\mathcal{P}_n, g)$ is an isometry if and only if there exists a matrix $M \in GL_n$ such that 
\begin{center}
$L= \Gamma_M$ \ \ or\ \ $L=\Gamma_M \circ \varphi$ \ \ or \ \ $L= \Gamma_M \circ \psi$ \ \ or \ \ $L= \Gamma_M \circ \varphi \circ \psi$.
\end{center}
\end{thm}

\begin{rem}\label{rem-dopo-isom-P}
When $n=2$, we have $\psi= \Gamma_W \circ \varphi = \varphi \circ \Gamma_W$ with $W$ as in \ref{rem-dopo-isom-SLP}. Hence in the previous Theorem there are only two mutually exclusive possibilities: $L= \Gamma_M$ and $L= \Gamma_M \circ \varphi$. 

If $n \ge 3$, since $\varphi$, $\psi$, $\varphi \circ \psi$ are not congruences, then the families of isometries, listed in \ref{isomP_n}, are mutually disjoint.
Therefore:

- if $n=2$ or $n$ is odd, then $\mathcal{I}(\mathcal{P}_n, g)$ is a Lie group with four connected components and with $\mathcal{I}^0(\mathcal{P}_n, g)$ isomorphic to $GL_n^+$ if $n$ is odd and $\mathcal{I}^0(\mathcal{P}_2, g)$ isomorphic to $GL_2^+/\{\pm\,I_2\}$;

- if $n\ne 2$ is even, then $\mathcal{I}(\mathcal{P}_n, g)$ is a Lie group with eight connected components and with $\mathcal{I}^0(\mathcal{P}_n, g)$ isomorphic to $GL_n^+/\{\pm I_n\}$.
\end{rem}

\begin{rem}\label{int-geom-isom}
Standard computations allow to obtain the following geometric descriptions of the isometries $\varphi$, $\psi$ and $\varphi \circ \psi$ by means of the results of \S 3 and \S 4:

- $\varphi$ is the symmetry with respect to $I_n$;

- $\psi$ is the orthogonal symmetry with respect to the hypersurface $SL\mathcal{P}_n$;

- $\varphi \circ \psi$ is the orthogonal symmetry with respect to the geodesic $\mathcal{R}=\{ t I_n : t \in \mathbb{R}, t>0\}$ (i.e. the geodesic through $I_n$ and orthogonal to $SL\mathcal{P}_n$).
\end{rem}

\begin{rem}
Let $\mathit{H}_n$ be the real manifold of positive definite hermitian matrices of order $n$. The tensor $g$ defines also on $\mathit{H}_n$ a structure of Riemannian manifold and $(\mathcal{P}_n, g)$ is one of its Riemannian submanifolds. From \cite{Mol2015} Thm.\,3 and from the previous \ref{isomP_n} we conclude that every isometry of $(\mathcal{P}_n, g)$ is restriction of an isometry of $(\mathit{H}_n, g)$.
\end{rem}

\end{document}